\documentclass[10pt,reqno]{amsart}
\usepackage{amsmath,amsthm,amssymb,mathtools,amscd}
\usepackage{latexsym,graphicx,color,indentfirst}
\usepackage{algorithmic}
\usepackage[linesnumbered,boxed,norelsize]{algorithm2e}
\usepackage{fancyhdr}
\usepackage{mathrsfs}
\usepackage{float}
\usepackage{float}
\usepackage{color}
\usepackage{listings}
\usepackage{multirow}
\usepackage{graphicx}
\usepackage{mathrsfs}
\usepackage{ifpdf}
\usepackage{epstopdf}
\usepackage{makecell,rotating,multirow,diagbox}
\usepackage[numbers,sort&compress]{natbib}

\newtheorem{theorem}{Theorem}[section]

\newtheorem{assumption}[theorem]{Assumption}

\newtheorem{remark}[theorem]{Remark}
\newtheorem{example}[theorem]{Example}

\date{\today}

\title[Simultaneously reconstructing potentials and internal sources for FSEs]{Simultaneously reconstructing potentials and internal sources for fractional Schr\"odinger equations}
\author{Xinyan Li$^{1}$}
\address{1.Research Center for Mathematics and Interdisciplinary Sciences, Shandong University, 266237, Shandong, P.R. China}
\email{\tt xinyanli@email.sdu.edu.cn}

\date{}

\begin{document}

\begin{abstract}
The inverse problems about fractional Calder\'on problem and fractional Schr\"odinger equations are of interest in the study of mathematics. In this paper, we propose the inverse problem to simultaneously reconstruct potentials and sources for fractional Schr\"odinger equations with internal source terms. We show the uniqueness for reconstructing the two terms under measurements from two different nonhomogeneous boundary conditions. By introducing the variational Tikhonov regularization functional, numerical method based on conjugate gradient method(CGM) is provided to realize this inverse problem. Numerical experiments are given to gauge the performance of the numerical method.
\end{abstract}

\maketitle

\section{\bf Introduction}
In the recent years, fractional Calder\'on problem has become a new developing research direction. Fractional Calder\'on problem was first proposed in \cite{GhoSal} to ask to determine the potential term in fractional Schr\"odinger equation from exterior measurements. The uniqueness, stability, optimal stability of fractional Calder\'on problem are proved in \cite{GhoSal,RulSala,RulSalb}, respectively. Further, Ghosh et al. show global uniqueness with a
single measurement in the fractional Calder\'on problem\cite{GhoshJFA}. There are also researches on the generalized form of \cite{GhoSal}, see \cite{GhoLin,LaiLina,LaiLinb,CovMoen,Covi22,BhaGho}, etc. In the real world, usually, the unknown source term may exist inside the object, which means that right-hand side of fractional Schr\"odinger equation may not be zero. In this situation, how to simultaneously determine the potential term and internal source term in fractional Schr\"odinger equation is an interesting problem.

In this paper, we consider the fractional Schr\"odinger equation with internal source
\begin{equation}\label{fseis}
\begin{cases}
((-\Delta)^s+q(x))u(x)=g(x),\quad &x\in\Omega,\\u(x)=f(x),\quad &x\in\Omega_e,
\end{cases}
\end{equation}
where $q\in L^{\infty}(\Omega)$, $g\in L^2(\Omega)$, $s\in(0,1)$, the open bounded set $\Omega\subset\mathbb{R}^n$, $n\geq1$, and $\Omega_e=\mathbb{R}^n\setminus\overline{\Omega}$. In the article, we will use integral fractional Laplacian described as
\begin{equation}\label{def1}
	(-\Delta)^{s}u(x)=c_{n,s}  \text{ P. V. }\int_{\mathbb{R}^n}\frac{u(x)-u(x')}{|x-x'|^{n+2s}}dx',
\end{equation}
where P. V. denotes the principal value integral, and $c_{n,s}$ is a constant depending on dimension $n$ and order $s$. According to \cite{Kwa17}, if $u$ belongs to the Schwartz space, then
\eqref{def1} is equivalent to the definition through Fourier transform introduced in \cite{GhoshJFA}. It will be shown that two combinations of different external sources $f$ and their corresponding external measurements $(-\Delta)^su$ can uniquely determine $q$ and $g$. Numerical method based on conjugate gradient method and numerical experiments are provided. Our conclusion is under the following condition:
\begin{assumption}\label{ass2}
If $u\in H^s(\mathbb{R}^n)=\left\{u \in L^2(\mathbb{R}^n); \int_{\mathbb{R}^n}\left(1+|\xi|^2\right)^s|\widehat{u}(\xi)|^2 d \xi<\infty\right\}$ is the solution of
\begin{equation*}
\left\{\begin{aligned}
((-\Delta)^s+q(x))u(x)=0,\quad &x\in\Omega,\\u=0,\quad &x\in\Omega_e,
\end{aligned}\right.
\end{equation*}
then $u\equiv0$. This is equivalent to the assumption that 0 is not a Dirichlet eigenvalue of $(-\Delta)^s+q$.
\end{assumption}
\begin{remark}
According to \cite[Theorem 4.1.]{SunDen} and the numerical scheme in \cite{Li23}, for $C_{\gamma,\beta}d_i<0,1\leq i\leq N-2$ always holds, we only needs $q_{\min}=\min_{1\leq i\leq N-1}q_i$ satisfies
$$C_{\gamma,s}d_0+q_j>0\text{ and }q_{\min}+\min_{i=1,\cdots,N-1}\sum_{j=1}^{N-1}C_{\gamma,\beta}d_{|i-j|}>0,$$
and then Assumption \ref{ass2} holds numerically.  For example, when $s=0. 4$, $L=1$, $N=128$, we only need $q_{\min}>-17. 9041$ after calculation.
\end{remark}

The outline of the paper is as follows. The uniqueness result of simultaneous reconstruction is given in Section 2. In Section 3, the variational Tikhonov regularization method is provided. Numerical experiments are shown in Section 4. Some concluding remarks are made in Section 5.

\section{\bf Uniqueness of simultaneous reconstruction}
Based on \cite[Lemma 2.3.]{GhoSal}, under Assumption \ref{ass2}, for any $f\in H^s(\mathbb{R}^n)$, $g\in (\tilde{H}^s(\Omega))^*$ with $\tilde{H}^s(\Omega)=\text{closure of }C_c^{\infty}(\Omega)\text{ in }H^s(\mathbb{R}^n)$, there is a unique weak solution $u\in H^s(\mathbb{R}^n)$ in \eqref{fseis}. The existence and uniqueness of weak solution is the basis of our problem. The following theorem illustrates that under two external sources and their corresponding observations, the potential $q$ and the source $g$ in \eqref{fseis} can be simultaneously determined.
\begin{theorem}
Let $u$ satisfy \eqref{fseis}, $\Omega\subset\mathbb{R}^n,n\geq1$ be bounded open set, $0<s<1$,
and $W_1,W_2\subset\Omega_e$ be open sets satisfying $\overline{\Omega}\cap \overline{W_1},\overline{\Omega}\cap \overline{W_2}=\emptyset$. Assume one of the following conditions holds:
\begin{enumerate}
\item $s\in[\frac{1}{4},1)$, $q\in L^{\infty}(\Omega)$, $g\in L^2(\Omega)$,
\item $s\in(0,1)$, $q\in C^0(\bar{\Omega})$, $g\in L^2(\Omega)$,
\end{enumerate}
and further, Assumption \ref{ass2} holds. Then for given $f,\tilde{f}\in\tilde{H}^s(W_1)$ being the boundary condition of
\begin{equation}\label{forw1}
\begin{cases}
((-\Delta)^s+q)u=g\quad &x\in\Omega,\\u=f\quad &x\in\Omega_e,
\end{cases}
\end{equation}
\begin{equation}\label{forw2}
\begin{cases}
((-\Delta)^s+q)\tilde{u}=g\quad &x\in\Omega,\\ \tilde{u}=\tilde{f}\quad &x\in\Omega_e
\end{cases}
\end{equation}
respectively, and satisfying $f-\tilde{f}\neq 0$ on every point at a nonzero measure subset of $W_1$, the potential $q$ and internal source $g$ can be uniquely determined by $\left.(-\Delta)^su\right|_{W_2}$ and $\left.(-\Delta)^s\tilde{u}\right|_{W_2}$.
\end{theorem}
\begin{proof}
Suppose two pairs of $q$ and $g$ denoting by $q_1,g_1$ and $q_2,g_2$ satisfy \eqref{fseis}. Then for external source $f$, there is
\begin{equation}\label{si1}
\begin{cases}
((-\Delta)^s+q_1)u=g_1\quad &x\in\Omega,\\u=f\quad &x\in\Omega_e,
\end{cases}
\end{equation}
\begin{equation}\label{si2}
\begin{cases}
((-\Delta)^s+q_2)\bar{u}=g_2\quad &x\in\Omega,\\ \bar{u}=f\quad &x\in\Omega_e,
\end{cases}
\end{equation}
with the observations $\left.(-\Delta)^su\right|_{W_2}=\left.(-\Delta)^s\bar{u}\right|_{W_2}$. Similarly, for the other external source $\tilde{f}$, the the two unknown pairs satisfy
\begin{equation}\label{si3}
\begin{cases}
((-\Delta)^s+q_1)\tilde{u}=g_1\quad &x\in\Omega,\\ \tilde{u}=\tilde{f}\quad &x\in\Omega_e,
\end{cases}
\end{equation}
\begin{equation}\label{si4}
\begin{cases}
((-\Delta)^s+q_2)\tilde{\bar{u}}=g_2\quad &x\in\Omega,\\ \tilde{\bar{u}}=\tilde{f}\quad &x\in\Omega_e,
\end{cases}
\end{equation}
respectively with $\left.(-\Delta)^s\tilde{u}\right|_{W_2}=\left.(-\Delta)^s\tilde{\bar{u}}\right|_{W_2}$.

Subtract \eqref{si3} from \eqref{si1}, and \eqref{si4} from \eqref{si2}, we respectively have
\begin{equation}\label{si5}
\begin{cases}
((-\Delta)^s+q_1)(u-\tilde{u})=0\quad &x\in\Omega,\\ u-\tilde{u}=f-\tilde{f}\quad &x\in\Omega_e,
\end{cases}
\end{equation}
\begin{equation}\label{si6}
\begin{cases}
((-\Delta)^s+q_2)(\bar{u}-\tilde{\bar{u}})=0\quad &x\in\Omega,\\ \bar{u}-\tilde{\bar{u}}=f-\tilde{f}\quad &x\in\Omega_e.
\end{cases}
\end{equation}
From \cite[Theorem 1.]{GhoshJFA}, $q$ can be uniquely determined from one combination of external source $f-\tilde{f}\in \tilde{H}^s(W_1)$ and observation $\left.(-\Delta)^s(u-\tilde{u})\right|_{W_2}=\left.(-\Delta)^s(\bar{u}-\tilde{\bar{u}})\right|_{W_2}$. Thus we obtain $q_1=q_2$. In the following proof, we will substitute $q_1$ and $q_2$ with $q$.

Further, subtract \eqref{si2} from \eqref{si1}, we have
\begin{equation}\label{si7}
\begin{cases}
((-\Delta)^s+q)(u-\bar{u})=g_1-g_2\quad &x\in\Omega,\\ u-\bar{u}=0\quad &x\in\Omega_e.
\end{cases}
\end{equation}
For both $u-\bar{u}$ and $\left.(-\Delta)^s(u-\bar{u})\right|_{W_2}$ are zero on the open set $W_2$, according to \cite[Theorem 1.2.]{GhoSal}, it holds that $u-\bar{u}\equiv0$, and then $g_1-g_2\equiv0$, $g_1=g_2$.
\end{proof}
\section{\bf Variational Tikhonov regularization method}
In this section, we give a numerical method to reconstruct the potential and internal source in \eqref{fseis} simultaneously. When $\Omega\subset\mathbb{R}$ satisfies strong local Lipschitz condition, it holds that $H^1(\Omega)\hookrightarrow L^{p}(\Omega),2\leq p\leq\infty$. For $f\in C_c^{\infty}(\Omega_e)$, $\mathrm{supp}(f)\subset W_1$, $W_1,W_2\subset\Omega_e$, $\overline{W_1}\cap\overline{\Omega}=\phi\,$, and $\overline{W_2}\cap\overline{\Omega}=\phi\,$, one can prove $\left. (-\Delta)^su\right|_{W_2}\in L^2(W_2)$ by verifying $\left. (-\Delta)^s(u-f)\right|_{W_2}\in L^2(W_2)$ and $\left. (-\Delta)^sf\right|_{W_2}\in L^2(W_2)$ respectively. Hence we could introduce the forward operator
$$F:H^1(\Omega)\times H^1(\Omega)\to L^2(W_2)\times L^2(W_2),(q,g)\mapsto((-\Delta)^su,(-\Delta)^s\tilde{u}).$$
In the following, we will assume $\Omega\subset\mathbb{R}$ is a bounded open set, $0<s<1$, $\Omega_e=\mathbb{R}\setminus\overline{\Omega}$, $W_1,W_2\subset\Omega_e$ are open sets and $\overline{\Omega}\cap \overline{W_1},\overline{\Omega}\cap\overline{W_2}=\emptyset\,$, $q$, $g$, $u$ satisfy the equation \eqref{fseis}, and the observation $\mathbf{h}^{\delta}(x)=(h^{\delta}(x),\tilde{h}^{\delta}(x))\in L^2(W_2)\times L^2(W_2)$ satisfies
$$\|\mathbf{h}^{\delta}(x)-F(q,g)\|_{L^2(W_2)}\leq\delta.$$
Moreover, a Tikhonov regularization functional is introduced by
\begin{equation*}
J(q,g)=\frac{1}{2}||F(q,g)-\mathbf{h}^{\delta}||_{L^2(W_2)}^2+\frac{\alpha}{2}||(q',g')||^2_{L^2(\Omega)},
\end{equation*}
where $\alpha$ is the regularization parameter, and $||(q',g')||^2_{L^2(\Omega)}$ is defined by $\int_{\Omega}( q')^2dx+\int_{\Omega}( g')^2dx$. Then solving the inverse problem is converted to solve the variational problem
\begin{equation*}
(q_{\alpha}^{\delta},g_{\alpha}^{\delta})=\underset{q\in H^1(\Omega),g\in H^1(\Omega)}{\arg\min}J(q,g).
\end{equation*}
\begin{remark}
By using similar methods in \cite{DinZhe,SunWei,ehn96}, the proof of the existence of the minimizer $(q_{\alpha}^{\delta},g_{\alpha}^{\delta})$ can be obtained. Nevertheless, it is hard for us to seek out the minimizer owing to machine precision, which can be handled by setting appropriate stopping rule to let $J(q_{\alpha}^{\delta})$ be closer to $\inf J(q)$.
\end{remark}
Here, we adopt conjugate gradient method to search the minimizer of $J(q,g)$, and the most crucial point is to find the descent gradient of $J(q,g)$. For simplicity, we restrict the admissible set to $Q_1=\{q\in L^2(\Omega):\|q\|_{H^1(\Omega)}<\infty,q=0\text{ near }\partial\Omega\}$.
Suppose $u_{q+\delta q,g+\delta g}-u_{q,g}=(u_q',u_g')\cdot(\delta q,\delta g)+o(\delta q+\delta g)$, $(\delta q,\delta g)\in Q_1\times Q_1$ and denote $\omega =(u_q',u_g')\cdot(\delta q,\delta g)$, where "$\cdot$" denotes inner product. The term $\tilde{\omega}$ can be defined likewise, and $\omega,\tilde{\omega}$ satisfy\\
\textbf{Sensitive Problems:}
\begin{equation}\label{sp11}
\left\{
\begin{aligned}
 ((-\Delta)^s+q(x))\omega(x)&=-\delta q\cdot u_{q,g}+\delta g, &x\in\Omega,\\
\omega(x)&=0, &x\in\Omega_e,
\end{aligned}
\right.
\end{equation}
\begin{equation}\label{sp21}
\left\{
\begin{aligned}
 ((-\Delta)^s+q(x))\tilde{\omega}(x)&=-\delta q\cdot \tilde{u}_{q,g}+\delta g, &x\in\Omega,\\
\tilde{\omega}(x)&=0, &x\in\Omega_e
\end{aligned}
\right.
\end{equation}
separately. Suppose that $v_{q,g}$, $\tilde{v}_{q,g}$ satisfy the following equations\\
\textbf{Adjoint Problems:}
\begin{equation}\label{ap11}
\left\{
\begin{aligned}
 ((-\Delta)^s+q(x))v(x)&=0, &x\in\Omega,\\
v(x)&=(-\Delta)^s\left.u_{q,g}\right|_{W_2}-h^{\delta}, &x\in\Omega_e,
\end{aligned}
\right.
\end{equation}
\begin{equation}\label{ap21}
\left\{
\begin{aligned}
 ((-\Delta)^s+q(x))\tilde{v}(x)&=0, &x\in\Omega,\\
\tilde{v}(x)&=(-\Delta)^s\left.\tilde{u}_{q,g}\right|_{W_2}-h^{\delta}, &x\in\Omega_e
\end{aligned}
\right.
\end{equation}
separately.
For one can notice that
\begin{equation}
\begin{aligned}
&(-q\omega-\delta q\cdot u_q+\delta g,v_q)_{\Omega}+((-\Delta)^s\left. \omega\right|_{W_2},(-\Delta)^s\left. u_q\right|_{W_2}-h^{\delta})_{W_2}\\
=&((-\Delta)^s\omega,v_q)_{\mathbb{R}^n}=(\omega,(-\Delta)^sv_q)_{\mathbb{R}^n}=(\omega,-qv_q)_{\Omega},
\end{aligned}
\end{equation}
the gradient of $J(q,g)$ can be obtained as
\begin{equation}\label{gradq1}
\mathbf{J}_{q,g}'=(u_{q,g}v_{q,g}+\tilde{u}_{q,g}\tilde{v}_{q,g}-\alpha q'',-v_{q,g}-\tilde{v}_{q,g}- \alpha g'').
\end{equation}

We use conjugate gradient method to search the minimizer of the functional $J(q,g)$. The CGM is efficient to deal with variational problems and has been widely used on various inverse problems\cite{DinZhe,SunWei,ZheDin}. Assume that $q^k$, $g^k$ are the $k$-th iteration approximate solution of $q(x)$ and $g(x)$. Then the updating formula of $q^k$ and $g^k$ is
$$(q^{k+1},g^{k+1})=(q^k,g^k)+\beta^k\mathbf{d}^k,$$
where $\beta^k$ is the step size, $\mathbf{d}^k$ is the descent direction in the $k$-th iteration and
\begin{equation}\label{descq1}
\mathbf{d}^k=-\mathbf{J}_{q^k,g^k}'+\gamma^k\mathbf{d}^{k-1},\mathbf{d}^0=-\mathbf{J}_{q^0,g^0}',
\end{equation}
where $\gamma^k$ is the conjugate coefficient calculated by
\begin{equation}\label{conjq1}
\gamma^k=\frac{||\mathbf{J}_{q^k,g^k}'||_{L^2(\Omega)}^2}{||\mathbf{J}_{q^{k-1},g^{k-1}}'||_{L^2(\Omega)}^2}, \gamma^0=0.
\end{equation}
The step size $\beta^k$ is given by
{\footnotesize{\begin{equation*} \hspace{-8mm}
\beta^k=\frac{\int_{W_2}((-\Delta)^su_{q^k,g^k}(x)-h^{\delta}(x))(-\Delta)^s\omega^k+((-\Delta)^s \tilde{u}_{q^k,g^k}(x)-\tilde{h}^{\delta}(x))(-\Delta)^s\tilde{\omega}^kdx+\alpha(((q^k)',(g^k)')\cdot (\mathbf{d}^k)')_{L^2(\Omega)}} {\int_{W_2}((-\Delta)^s\omega^k)^2+((-\Delta)^s\tilde{\omega}^k)^2dx+\alpha( (\mathbf{d}^k)', (\mathbf{d}^k)')_{L^2(\Omega)}},
\end{equation*}}}
such that $\frac{\partial\mathbf{J}}{\partial\beta^k}\approx0$, where we denote $((u_1,v_1),(u_2,v_2))_{L^2(\Omega)}=\int_{\Omega}u_1u_2dx+\int_{\Omega}v_1v_2dx$. The iteration steps are given by
\begin{enumerate}
\item Initialize $q^0,g^0$, and set $k=0$, $\mathbf{d}^{0}=-\mathbf{J}_{q^0,g^0}'$;
\item Solve the forward problems (\ref{forw1}),(\ref{forw2}), where $q=q^k,g=g^k$, and determine the residuals $(-\Delta)^s\left.u\right|_{W_2}-h^{\delta}$,$(-\Delta)^s\left.\tilde{u}\right|_{W_2}-\tilde{h}^{\delta}$;
\item Solve the adjoint problems (\ref{ap11}), (\ref{ap21}), and determine the gradient $\mathbf{J}_{q^k,g^k}'$ through (\ref{gradq1});
\item Calculate the conjugate coefficient $\gamma^k$ by \eqref{conjq1} and the descent direction by (\ref{descq1});
\item Solve the sensitive problems (\ref{sp11}), (\ref{sp21}) and obtain $\omega^k$, $\tilde{\omega}^k$ with $(\delta q,\delta g)=\mathbf{d}^k$;
\item Using the step size updating formula below \eqref{conjq1}, update the step size $\beta^k$;
\item Update $q^k$, $g^k$ by $(q^{k+1},g^{k+1})=(q^k,g^k)+\beta^k\mathbf{d}^k$;
\item Increase $k=k+1$, return to step (2), and repeat the above procedure until a stopping criterion is satisfied.
\end{enumerate}

\section{\bf Numerical experiments}
Let the observation be
\begin{equation*}
\mathbf{h}^{\delta}=(\left. (-\Delta)^su\right|_{W_2},\left. (-\Delta)^s\tilde{u}\right|_{W_2})+\delta (2rand(N)-1,2rand(N)-1),
\end{equation*}
where $N$ is the number of discrete points in $W_2$. In the following examples, we use high-precision grids to calculate the observation and solve the fractional Schr\"odinger equations by methods in \cite{Li23,DuoWyk}.
\begin{example}\label{ex1si}
Let $q(x) = \sin(x), g(x)=\cos(x) ,\Omega=(-1,1),s=0.4$, $W_2=(-3,-1-\epsilon)\cup(1+\epsilon,3)$, $0<\epsilon\ll1$, $W_1=W_2$, $f, \tilde{f}$ be $1$ and $e^{-x^2}$ polished function on $W_1$ respectively, and the stopping rule in the iteration algorithm is
\begin{equation*}
E_{k+1}=||F(q^k,g^k)-\mathbf{h}^{\delta}||_{L^2(W_2)}^2\leq4\times\delta^2.
\end{equation*}
We choose $\alpha=\delta^2$, Figure \ref{qg}(a)-(b) report the numerical results of $q$ and $g$, where we can see that when $\delta$ increases, the reconstruction result will get worse.
\begin{figure}[htbp] 
\begin{minipage}[b]{.48\linewidth}
  \centering
  \centerline{\includegraphics[width=6.0cm]{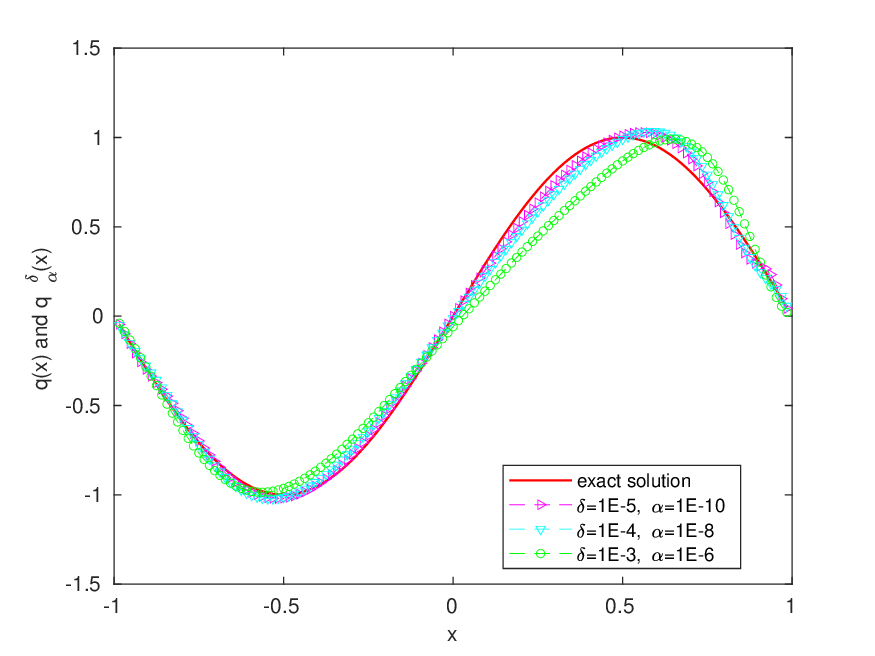}}
  \centerline{(a)}\label{qg}\medskip
\end{minipage}
\hfill              
\begin{minipage}[b]{0.48\linewidth}
  \centering
  \centerline{\includegraphics[width=6.0cm]{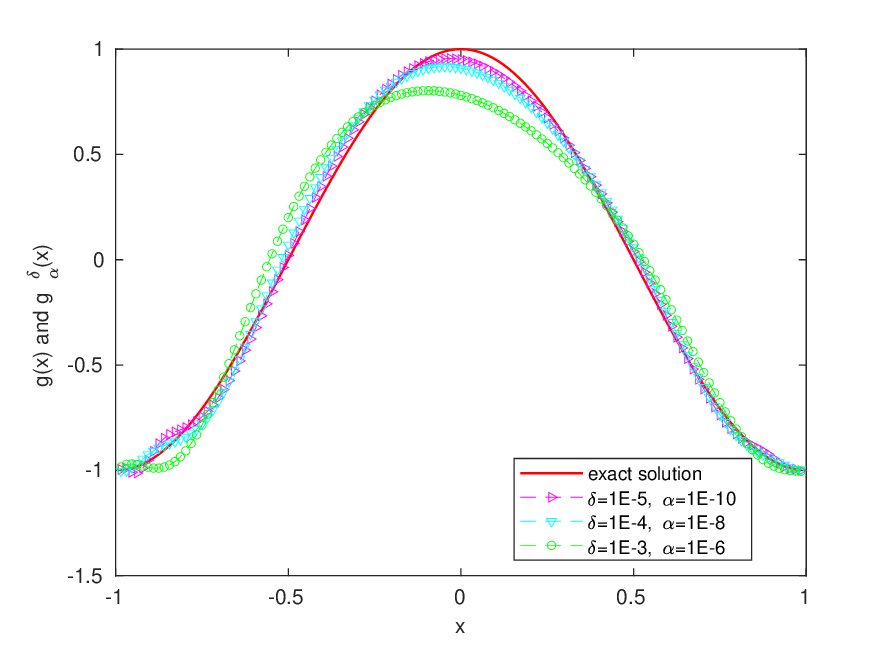}}
  \centerline{(b)}
  \label{gq}\medskip
\end{minipage}
\caption{The numerical results for Example \ref{ex1si} with different noise levels, (a) the reconstruction of $q$; (b) the reconstruction of $g$.}
\end{figure}
\end{example}
\begin{example}\label{ex2si}
Let $q(x) = 1-x^2, g(x)=1-x^4 ,\Omega=(-1,1),s=0.4$,  $W_2=(-3,-1-\epsilon)\cup(1+\epsilon,3)$, $0<\epsilon\ll1$, $W_1=W_2$, $f, \tilde{f}$ be $1$ and $e^{-x^2}$ polished function on $W_1$ respectively, and the stopping rule in the iteration algorithm is
\begin{equation*}
E_{k+1}=||F(q^k,g^k)-\mathbf{h}^{\delta}||_{L^2(W_2)}^2\leq40\times\delta^2.
\end{equation*}
We choose $\alpha=\delta^2$, Figure \ref{qg2}(a)-(b) show the reconstruction results of $q$ and $g$, from which we can see that when $\delta$ increases, the reconstruction result will become untruthful.
\begin{figure}[htbp] 
\begin{minipage}[b]{.48\linewidth}
  \centering
  \centerline{\includegraphics[width=6.0cm]{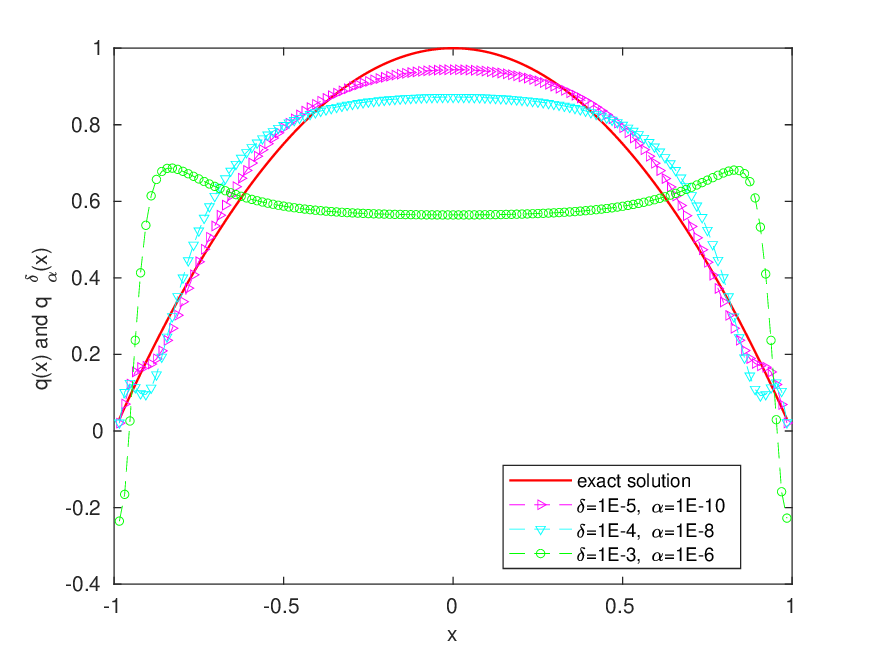}}
  \centerline{(a)}\label{qg2}\medskip
\end{minipage}
\hfill              
\begin{minipage}[b]{0.48\linewidth}
  \centering
  \centerline{\includegraphics[width=6.0cm]{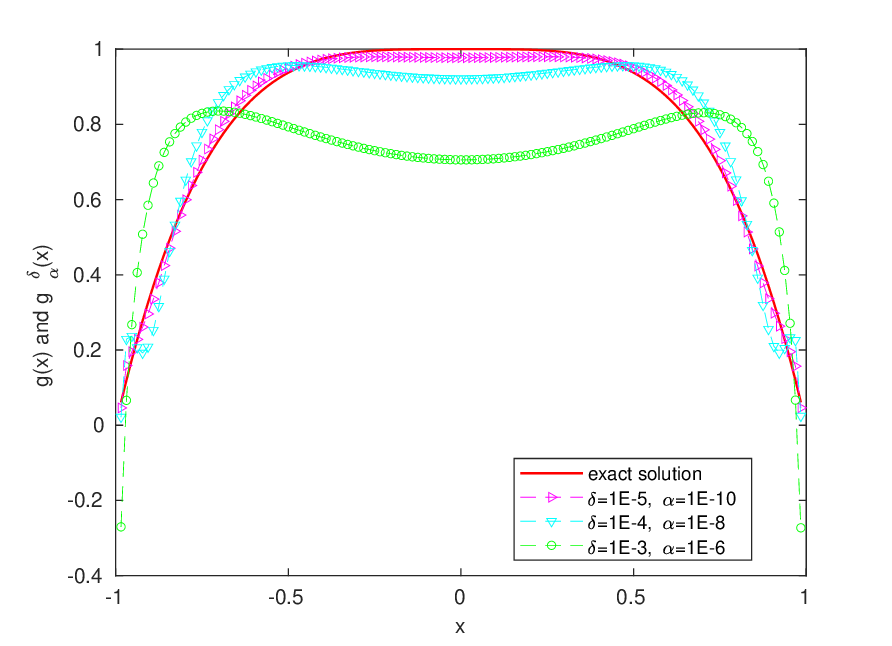}}
  \centerline{(b)}
  \label{gq2}\medskip
\end{minipage}
\caption{The numerical results for Example \ref{ex2si} with different noise levels, (a) the reconstruction of $q$; (b) the reconstruction of $g$.}
\end{figure}
\end{example}

\section{\bf Conclusions}
The paper investigates simultaneous reconstruction of potentials and sources for fractional Schr\"odinger equations. We obtain the uniqueness for reconstructing these two terms under measurements from two different exterior Dirichlet boundary conditions. The problem can be realized numerially by formulating a variational Tikhonov regularization method, whose solution can be searched by conjugate gradient method. The numerical experiments show the effectiveness of our proposed method.

  \end{document}